\documentclass[12pt]{article}
\usepackage{e-jc}

\usepackage{color}
\usepackage{epsfig}
\usepackage{amsmath}
\usepackage{verbatim}
\usepackage{subfig} 
\newcommand{\subfigure}{\subfloat}

\usepackage{amsthm}
\usepackage{amssymb}
\usepackage{graphicx}
\usepackage{multirow}
\usepackage{array}
\usepackage{url}

\newtheorem{conjecture}{Conjecture}
\newtheorem{lemma}{Lemma}
\newtheorem{theorem}{Theorem}
\newtheorem{corollary}{Corollary}

\newtheorem{varthm}{Theorem}

% ...

\newcommand{\lref}[1]{Lemma \ref{lem:#1}}
\newcommand{\cref}[1]{Corollary \ref{cor:#1}}
\newcommand{\conref}[1]{Conjecture \ref{conj:#1}}
\newcommand{\eref}[1]{Equation (\ref{eq:#1})}
\newcommand{\sref}[1]{Section \ref{sec:#1}}
\newcommand{\tref}[1]{Theorem \ref{thm:#1}}
\newcommand{\tabref}[1]{Table \ref{tab:#1}}
\newcommand{\fref}[1]{Figure \ref{fig:#1}} 

\newcommand{\T}{$\mathsf T$}

\date{\dateline{March 4, 2011}{XX}\\\small Mathematics Subject Classification: 05B45,05A19}

\title{Auspicious tatami mat arrangements}

\author{Alejandro Erickson\\[-0.8ex]
\small \texttt{ate@uvic.ca}\\[-0.8ex]
\small Department of Computer Science\\
\and
Frank Ruskey\\[-0.8ex]
\small \texttt{ruskey@uvic.ca}\\[-0.8ex]
\small Department of Computer Science\\
\and
Jennifer Woodcock\\[-0.8ex]
\small \texttt{jwoodcoc@uvic.ca}\\[-0.8ex]
\small Department of Computer Science\\
\small University of Victoria \\[-0.8ex]
\small PO BOX 3055, STN CSC, Victoria BC, V8W 3P6, Canada\\
\and
Mark Schurch\\[-0.8ex]
\small \texttt{mschurch@uvic.ca}\\[-0.8ex]
\small Department of Mathematics and Statistics\\
\small University of Victoria\\[-0.8ex]
\small PO BOX 3060, STN CSC, Victoria BC, V8W 3R4, Canada\\
}

\begin{document}

\maketitle

\begin{abstract}
  An \emph{auspicious tatami mat arrangement} is a tiling of a rectilinear
    region with two types of tiles, $1 \times 2$ tiles (dimers) and $1 \times 1$ tiles (monomers).
  The tiles must cover the region and satisfy the constraint that no
    four corners of the tiles meet; such tilings are called \emph{tatami tilings}.
  The main focus of this paper is when the rectilinear region is a rectangle.
  We provide a structural characterization of rectangular tatami tilings and use it
    to prove that the tiling is completely determined by the tiles that are on
    its border.
  We prove that the number of tatami tilings of an $n \times n$ square with $n$ monomers is
    $n2^{n-1}$.
  We also show that, for fixed-height, the generating
    function for the number of tatami tilings of a rectangle is a
    rational function, and outline an algorithm that produces the generating function.
\end{abstract}

\textbf{Keywords:} {tatami, monomer-dimer tiling, rational generating function}

\section{What is a tatami tiling?}

Traditionally, a tatami mat is made from a rice straw core, with a
covering of woven soft rush straw.  Originally intended for nobility
in Japan, they are now available in mass-market stores.  The typical
tatami mat occurs in a $1 \times 2$ aspect ratio and various
configurations of them are used to cover floors in houses and temples.
By parity considerations it may be necessary to introduce mats with a
$1 \times 1$ aspect ratio in order to cover the floor of a room.  Such
a covering is said to be ``auspicious'' if no four corners of mats
meet at a point.  Hereafter, we only consider auspicious
arrangements, since without this constraint the problem is the
classical and well-studied dimer tiling problem (\cite{KenyonAkounkov},
\cite{Stanley}).  Following Knuth (\cite{Knuth}), we will call the
auspicious tatami arrangements, \emph{tatami tilings}.  The
fixed-height enumeration of tatami tilings that use only dimers (no
monomers) was considered in \cite{RuskWood}, and results for the single
monomer case were given in \cite{Alhazov}.

\begin{figure}
  \begin{center}
    \subfigure[]{\includegraphics{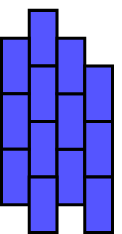}}
    \hspace{0.5in}
    \subfigure[]{\includegraphics{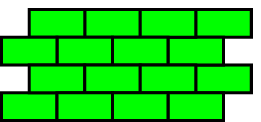}}
    \hspace{0.5in}
    \subfigure[]{\includegraphics{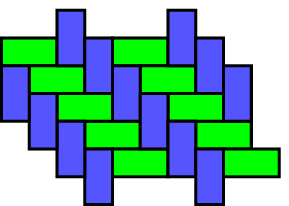}}    

\caption{(a) Vertical bond pattern. (b) Horizontal bond pattern. (c) Herringbone pattern.}
    \label{fig:patterns}
  \end{center}
\end{figure}

Perhaps the most commonly occurring instance of tatami tilings is in
paving stone layouts of driveways and sidewalks, where the most
frequently used paver has a rectangular shape with a $1 \times 2$
aspect ratio.  Two of the most common patterns, the ``herringbone''
and the ``running bond,'' shown in \fref{patterns}, have the tatami
property.  Consider a driveway of the shape in \fref{driveway}.  How
can it be tatami tiled with the least possible number of monomers?
The answer to this question could be interesting both because of
aesthetic appeal, and because it could save work, since to make a
monomer a worker typically cuts a $1 \times 2$ paver in half.

\begin{figure}
  \begin{center}
\includegraphics[width = 0.75\textwidth]{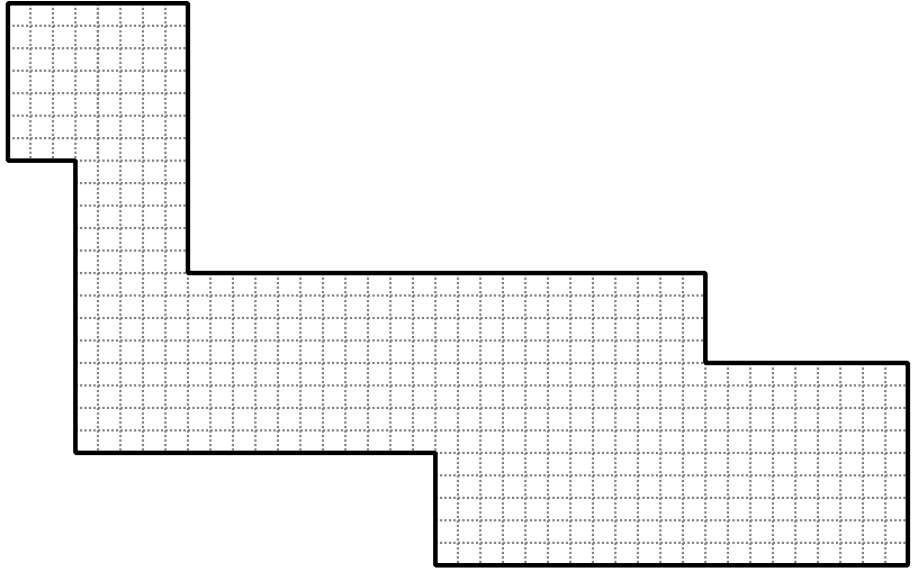}
\caption{What is the least number of monomers among all tatami tilings
  of this region?  The answer is provided at the end of the paper in
  \fref{solution}.}
\label{fig:driveway}
  \end{center}
\end{figure}

Before attempting to study tatami tilings in general orthogonal
regions it is crucial to understand them in rectangles, and our
results are primarily about tatami tilings of rectangles. 

\subsection{Outline}
In Section 2 we determine the structure of tatami tilings in a
rectangle.  Our structural characterization has important algorithmic
implications, for example, it reduces the size of the description of a
tiling from $\Theta(rc)$ to $O(\max\{r,c\})$ and may be used to
generate tilings quickly.  The three theorems in \sref{results} are
the main results of the paper and are also stated here.  The first of
these concerns the maximum possible number of monomers.  Let
$T(r,c,m)$ be the number of tilings of the $r\times c$ grid, with $m$
monomers (and the other tiles being horizontal or vertical dimers).

\begin{varthm}
  If $T(r,c,m)>0$, then $m$ has the same parity as $rc$ and
  $m\le\max(r+1,c+1)$.
\end{varthm}

Following this we prove a counting result for
maximum-monomer tilings of square grids.

\begin{varthm}
  The number of $n\times n$ tilings with $n$ monomers, $n2^{n-1}$.
\end{varthm}

Our final result concerns fixed-height tilings with an unrestricted number of monomers.

\begin{varthm}
  For a fixed number of rows $r$, the ordinary generating function of
  the number of tilings of an $r \times n$ rectangle is a rational
  function.
\end{varthm}

We also provide an algorithm which outputs this generating function
for a given $r$ and explicitly give the generating function for $r=1,
2$ and $3$, along with the coefficients of the denominator for $1 \le
r \le 11$.  In Section 4 we return to the question of tatami tiling
general orthogonal regions and introduce the ``magnetic water strider
problem'' along with additional conjectures and open problems.

\section{The structure of tatami tilings: \T-diagrams}

We show that all tatami tilings have an underlying structure which
partitions the grid into regions, where each region is filled with
either the vertical or horizontal running bond pattern (or is a
monomer not touching the boundary).  For example, in
\fref{comprehensiveTiling} there are 11 regions, including the
interior monomer.  We will describe this structure precisely and prove
some results for tilings of rectangular grids.

\begin{figure}[ht]
\centering
 \includegraphics{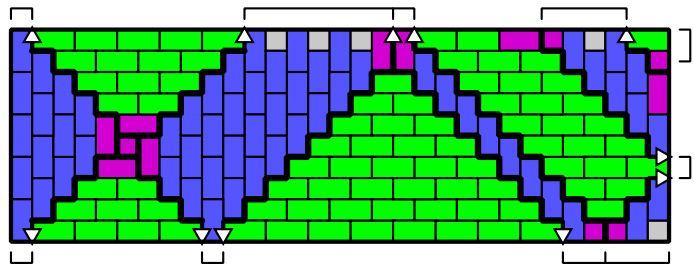}
 \caption{A tiling showing all four types of sources.  Coloured in
   magenta, from left to right they are, a clockwise vortex, a vertical
   bidimer, a loner, a vee, and another loner.  Jagged edges are
   indicated by brackets.}
   \label{fig:comprehensiveTiling}
 \end{figure}

\begin{figure}[ht]
  \centering
  \subfigure[A \emph{loner} source.]{\includegraphics{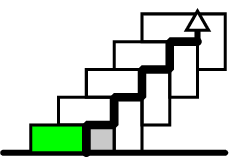}
  \label{fig:lonersource}}
  \hspace{0.5in}
  \subfigure[A \emph{vee} source.]{\includegraphics{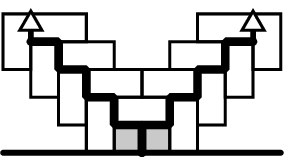}
  \label{fig:veesource}}
\caption{These two types of sources must have their coloured tiles on
  a boundary, as shown, up to rotational symmetry.}
  \label{fig:lonerveesources}
\end{figure}

\begin{figure}[ht]
  \centering
  \includegraphics{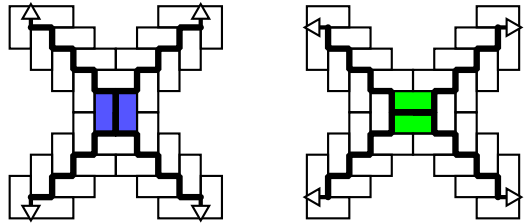}
  \caption{A vertical and a horizontal \emph{bidimer} source.  A
    bidimer may appear anywhere in a tiling provided that the coloured
  tiles are within the boundaries of the grid.}
  \label{fig:bidimersource}
\end{figure}

\begin{figure}[ht]
  \centering
  \includegraphics{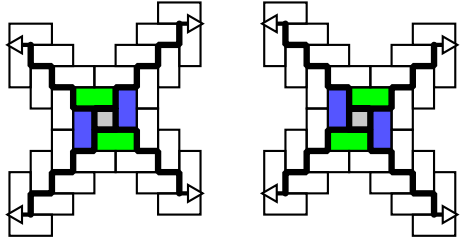}
  \caption{A counter clockwise and a clockwise \emph{vortex} source.   A
    vortex may appear anywhere in a tiling provided that the coloured
  tiles are within the boundaries of the grid.}
  \label{fig:vortexsource}
\end{figure}

Wherever a horizontal and vertical dimer share an
edge \includegraphics[scale=0.5]{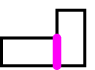}, either the placement of
another dimer is forced to preserve the tatami condition, or the tiles
make a \T~with the boundary of the
grid \includegraphics[scale=0.5]{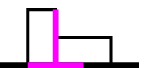}.  In the former case,
the placement of the new dimer again causes the sharing of an
edge \includegraphics[scale=0.5]{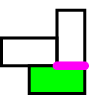}, and so
on \includegraphics[scale=0.5]{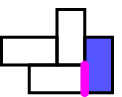}, until the boundary
is reached.

The successive placement of dimers, described above gives rise to
  skinny herringbone formations, which we call \emph{rays}.  They
  propagate from their source to the boundary of the grid and cannot
  intersect one another.  Between the rays, there are only vertical or
  horizontal running bond patterns.  The intersection of a running
  bond with the boundary is called a \emph{segment}.  This segment is
  said to be \emph{jagged} if it consists of alternating monomers and
  dimers orthogonal to the boundary; otherwise it is said to
  be \emph{smooth} because it consists of dimers that are aligned with
  the boundary.  Every jagged segment is marked with square brackets
  in \fref{comprehensiveTiling}.

We know that a ray, once it starts, propagates to the boundary.  But
how do they start?  In a rectangular grid, we will show that a ray
starts at one of four possible types of
\emph{sources}.  In our discussion we use inline diagrams to
depict the tiles that can cover the grid squares at the start of a
ray.
We need not consider the case where the innermost square (denoted by the circle)
\includegraphics{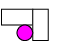} is covered by a vertical
dimer \includegraphics{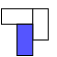} because this would move the
start of the ray.

If it is covered by a horizontal dimer
\includegraphics{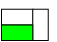}, the source, which consists of the
two dimers that share a long edge, is called a \emph{bimer}.
Otherwise it is covered by a monomer \includegraphics{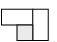}
in which case we consider the grid square beside
it \includegraphics{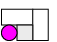}.  If it is covered by a monomer
the source is called a \emph{vee}
\includegraphics{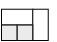}; if it is covered by a vertical dimer
the source is called a \emph{vortex} \includegraphics{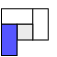};
if it is covered by a horizontal dimer it is called a \emph{loner}
\includegraphics{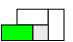}.  Each of these four types of sources forces at least
one ray in the tiling and all rays begin at either a bidimer, vee,
vortex or loner.  The different types of features are depicted in
Figures \ref{fig:lonerveesources}-\ref{fig:vortexsource}.

The coloured tiles in Figures
\ref{fig:lonerveesources}-\ref{fig:vortexsource} characterize the four
types of sources.  A bidimer or vortex may appear anywhere in a
tiling, as long as the coloured tiles are within its boundaries.  The
vees and loners, on the other hand, must appear along a boundary, as
shown in \fref{lonerveesources}.


The collection of bold staircase-shaped curves in each of the four types of
source-ray drawings in
Figures \ref{fig:lonerveesources}-\ref{fig:vortexsource}, is called a
\emph{feature}.  These features do not intersect when drawn on a
tatami tiling because rays cannot intersect. A \emph{feature-diagram}
refers to a set of non-intersecting features drawn in a grid.  Not
every feature-diagram admits a tatami tiling; those that do are called
\emph{\T-diagrams}.  See \fref{featureTdiagram}.

\begin{figure}[ht]
 \centering
\includegraphics[width=5.0in]{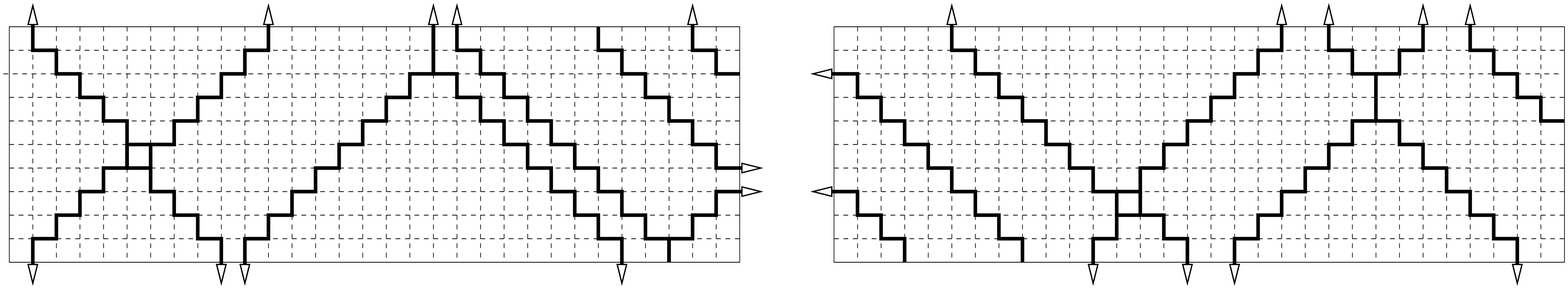}

\hspace{0.0in} (a) \hspace{2.0in} (b)

\caption{(a) The \T-diagram of \fref{comprehensiveTiling}.
  (b) A feature diagram that is not a \T-diagram.}
\label{fig:featureTdiagram}
\end{figure}

Recall that a tatami tiling consists of regions of horizontal and
vertical running bond patterns.  A feature-diagram is a \T-diagram if
and only if each pair of rays bounding the same region admit bond
patterns of the same orientation and the distance between them has the
correct parity.  The precise conditions are stated in \lref{conditions}.



Features decompose into four types of rays, to which we assign the
symbols $NW$, $NE$, $SW$, and $SE$, indicating the direction of
propagation.  Two rays are said to be \emph{adjacent} if they can be
connected by a horizontal or vertical line segment which intersects no
other ray.  If $(\alpha, \beta)$ is an adjacent pair, then $\alpha$ is
on the left when considering horizontally adjacent pairs and on the
bottom when considering vertically adjacent pairs.


\begin{lemma}

   A feature diagram is a \T-diagram if and only if the following four
   conditions hold.

   Horizontal Conditions:
   \begin{enumerate}
   \item[(H1)] There are no horizontal $(\alpha E,\beta
     E)$-adjacencies, nor are there horizontal $(\alpha W, \beta
     W)$-adjacencies, where $\alpha$ and $\beta$ are either $N$ or
     $S$ (\fref{badPairs}); 
   \item[(H2)] all distances are even, except for horizontal
     $(NE,NW)$-distances and horizontal $(SE,SW)$-distances, which are
     odd
     (\fref{parities}). 
   \end{enumerate}
   
   Vertical Conditions:
   \begin{enumerate}
   \item[(V1)] There are no vertical $(S\alpha ,S\beta )$-adjacencies,
     nor are there any vertical $(N\alpha , N\beta )$-adjacencies,
     where $\alpha$ and $\beta$ are either $E$ or $W$;
   \item[(V2)] all distances are even, except for vertical\\
     $(NW,SW)$-distances and vertical $(NE,SE)$-distances, which are
     odd.
   \end{enumerate}
 \label{lem:conditions}
 \end{lemma}

 \begin{figure}[ht]
   \centering 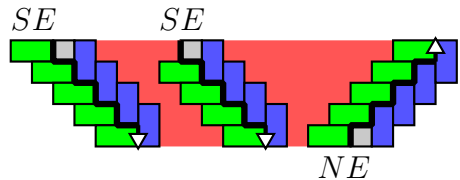
   \caption{Incompatible pairs of adjacent rays.  The
     region between the adjacent rays would have to contain both
     horizontal and vertical dimers.}
   \label{fig:badPairs}
 \end{figure}
 \begin{figure}[ht]
   \centering
   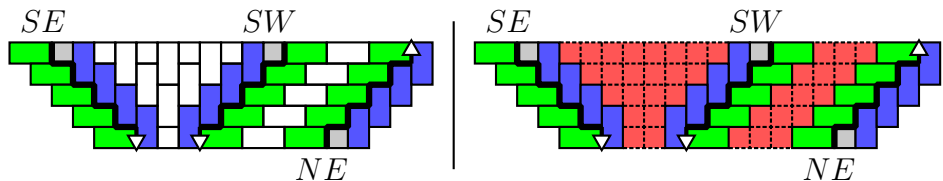
   \caption{If the size of the gap between adjacent rays has the
     correct parity then it can be properly tiled, as shown on the
     left.  On the right, the red regions cannot be tiled because the
     gaps have the wrong parities.}

   \label{fig:parities}
\end{figure}

This characterization has some implications for the space and time
complexity of a tiling.

\begin{lemma}
  \label{lem:conditionbenefits}
  Let $G$ be an $r\times c$ grid, with $r<c$.
\begin{itemize}
\item[(i)] The storage requirement for a tatami tiling of $G$ is $O(c)$; that is,
  a tatami tiling can be recovered from $O(c)$ integers, each of size at most $c$.
\item[(ii)] A tatami tiling of $G$ is uniquely determined by the tiles on its boundary.
\item[(iii)] Whether a feature diagram in $G$ is a \T-diagram can be
  determined in time $O(c)$.
\end{itemize}
\end{lemma}
\begin{proof}
  To prove (i), notice that a non-trivial \T-diagram defines a tiling
  uniquely.  In the \emph{trivial} case there are no features and
  exactly four possible running bond configurations, two horizontal
  and two vertical.  Otherwise, each feature can be stored as a pair
  of coordinates and a type.  It is not possible to have more
  than $O(c)$ compatible features in a \T-diagram in $G$, so at most
  $O(c)$ integers of size at most $c$ are needed, proving (i).

  To prove (ii), we need to show that we can recover the \T-diagram
  from the tiles that touch the boundary.  Those portions of the
  \T-diagram corresponding to vees and loners, as well as bidimers
  whose source tiles are both on the
  boundary \includegraphics[scale=0.5]{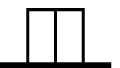}, are easy to
  recover.  The black rays in \fref{tilingFromBorder} show their
  recovery.  Imagine filling in the remaining red rays, whose ends
  look like \includegraphics[scale=0.5]{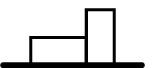}, by following them
  na\"ively, backwards from their endings to the boundary.  The ends
  of the four rays emanating from a bidimer or vortex will always form
  exactly one of the four patterns illustrated in
  \fref{veesHamsFromBorder}; in each case, it is straightforward to
  recover the position and type of source.  This proves (ii).

\begin{figure}[ht]
\centering
\includegraphics[scale=1.0]{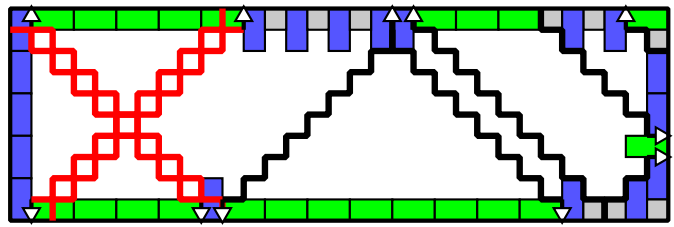}

\caption{The same tiling as in \fref{comprehensiveTiling} with only
  the boundary tiles showing.  Rays emanating from sources on the
  boundary are in black and otherwise, they are drawn na\"ively in
  red, to be matched with a candidate source from
  \fref{veesHamsFromBorder}.}
 \label{fig:tilingFromBorder}
\end{figure}

\begin{figure}[ht]
\centering
\includegraphics[scale=1.0]{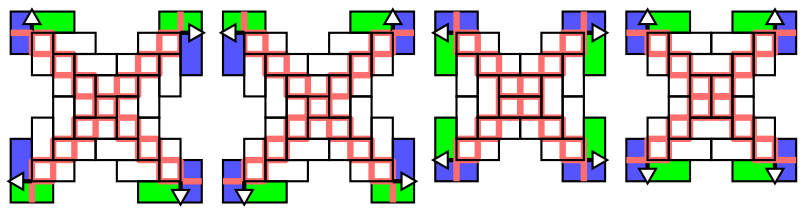}

\caption{The four types of vortices and bidimers are recoverable
  from the ends of their rays, at the boundary of the grid.  Given the
  ends of rays, they can be extended na\"ively to form one of the two
  patterns in red.  One occurs only for bidimers and the other for
  vortices.  The orientation of the source is determined by the ends of
  the rays.}
    \label{fig:veesHamsFromBorder}
  \end{figure}

  Claim (iii) is true provided that \lref{conditions} only needs to be
  applied to $O(c)$ ray-adjacencies.  Notice that a pair of rays can
  be adjacent and yet not be adjacent on the boundary.  For example,
  it happens in \fref{featureTdiagram}.

  Each ray bounds exactly two regions, each of which is bounded by at
  most three other rays, and two rays must bound the same region to be
  adjacent.  Thus, a ray is adjacent to at most six other rays.  Let
  the ray-adjacencies be the edges of a graph $G = (V,E)$ whose vertex
  set is the set of rays, so that $G$ has maximum degree at most
  $6$. Therefore, the number of ray-adjacencies, $|E|$, and hence
  applications of \lref{conditions}, is linear in the number of rays,
  $|V|$, which is at most four times the number of features, which is
  in $O(c)$.  This proves (iii).
\end{proof}

The \T-diagram structure is a useful tool for enumerating and generating
tatami tilings as will be illustrated in the following sections.

\section{Counting results}
\label{sec:results}

Let $T(r,c,m)$ be the number of tatami tilings of a rectangular grid
with $r$ rows, $c$ columns, and $m$ monomers.  Also, $T(r,c)$ will
denote the sum
\begin{align*}
  T(r,c) = \sum_{m \ge 0} T(r,c,m).
\end{align*}

We begin by giving necessary conditions for $T(r,c,m)$ to be non-zero.

\begin{theorem}\label{thm:maxmon}
  If $T(r,c,m)>0$, then $m$ has the same parity as $rc$ and
  $m\le\max(r+1,c+1)$.
\end{theorem}
\begin{proof}
  Let $r,c$ and $m$ be such that $T(r,c,m)>0$ and let $d$ be the
  number of grid squares covered by dimers in an $r\times c$ tatami
  tiling so that $m=rc-d$.  Since $d$ is even, $m$ must have the same
  parity as $rc$.

  It suffices to assume that $r\le c$, and prove that $m\le c+1$.  The
  proof proceeds in two steps.  First, we will show that a monomer on
  a vertical boundary of any tiling can be \emph{mapped} to the top or
  bottom, without altering the position of any other monomer.  Then we
  can restrict our attention to tilings where all monomers appear on
  the top or bottom boundaries, or in the interior.  Secondly, we will
  show that there can be at most $c+1$ monomers on the combined
  horizontal boundaries.


  Let $T$ be a tatami tiling of the $r\times c$ grid with a monomer
  $\mu$ on the left boundary, touching neither the bottom nor the top
  boundary.  The monomer $\mu$ is (a) part of a vee or a loner, or (b)
  is on a jagged segment of a region of horizontal bond.  Define a
  \emph{diagonal} to be $\mu$ together with a set of dimers in this
  region which form a stairway shape from $\mu$ to either the top or
  bottom of the grid as shown in purple in \fref{moveSide}.  If such a
  diagonal exists, a \emph{diagonal flip} can be applied, which
  changes the orientation of its dimers and maps $\mu$ to the other
  end of the diagonal.  In case (a) a diagonal clearly exists since
  it is a source and its ray will hit a horizontal boundary because
  $r \le c$.

  If $\mu$ is on a jagged segment, then we argue by contradiction.
  Suppose neither diagonal exists, then they must each be
  impeded by a distinct ray.  Such rays have this horizontal region to
  the left so the upper one is directed $SE$ and the lower $NE$ and
  they meet the right boundary (before intersecting).  Referring to
  \fref{noMoveSide},
  \begin{align*}
    \alpha + \beta + j =& \gamma + \delta + 1\le r  \\
    \le& c \le c' =  \alpha + \gamma = \beta + \delta,
  \end{align*}
  where $j$ is some odd number.
  Thus $\alpha + \beta + j \le \alpha + \gamma$ implying that $\beta <
  \gamma$.  On the other hand,
  \[
  \gamma + \delta +1  = r \le c \le c' = \beta + \delta
  \]
  implies that $\gamma < \beta$, which is a contradiction.  Therefore
  at least one of the diagonals exists and the monomer can be mapped
  to a horizontal boundary.

\begin{figure}
  \centering

  \subfigure[A diagonal flip.]{\includegraphics[scale = 0.6]{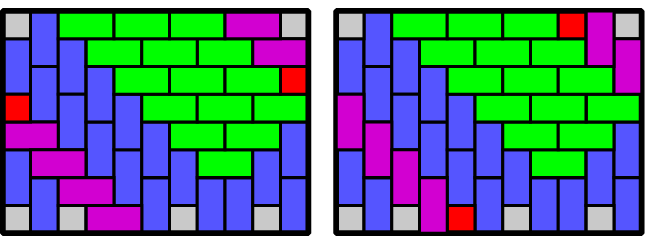}
  \label{fig:moveSide}}
  \def\svgwidth{1.5in}
  \subfigure[The case for vees.]{\includegraphics[width = 1.5in]{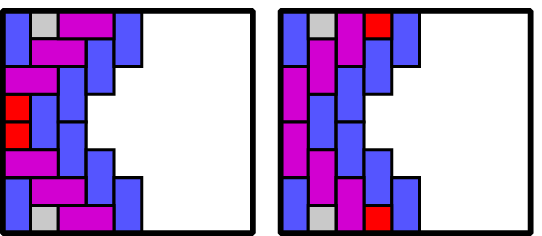}
    \label{fig:moveSideVee} } \def\svgwidth{1in}
  \subfigure[]{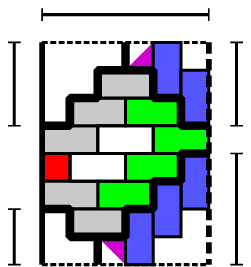
\label{fig:noMoveSide}  }
\caption{(c) If both diagonals are blocked, then $c<r$.  The tiling is
  at least this tall and at most this wide.}
\end{figure}

We may now assume that there are no monomers strictly on the vertical
boundaries of the tiling, and therefore all monomers are either in the
top or bottom rows or in vortices.  Let $v$ be the number of vortices.
Encode the bottom and top rows of the tiling by length $c$ binary
sequences $P$ and $Q$, respectively.  In the sequences, $1$s represent
monomers and $0$s represent squares in dimers.

If $Q$ contains a consecutive pair of $1$s, this represents a vee in
the top row; the vee has a region of horizontal dimers directly below
it.  This region of horizontal bond must reach the bottom row
somewhere, otherwise, by an argument similar to one given previously,
we would have $c<r$ (see \fref{maxMonLimits}).  Therefore, there must
be a consecutive pair or $0$s in $P$ unique to these $1$s in $Q$.
Modify $Q$ so that the $11$ becomes a $1$, and modify $P$ so that the
$00$ becomes $010$.  Do this for each consecutive pair of $1$s in $P$
and $Q$, as illustrated in \fref{maxMonSeq}.  The updated $P$ and $Q$
sequences contain no consecutive pairs of $1$s, but the total number
of $1$s remains unchanged.

Now we show that for each of the $v$ vortices, a $00$ can be removed
from each of $P$ and $Q$ and there will still be no consecutive pairs
of $1$s in the new sequences.  Each vortex generates 4 rays; at least
one of these rays will hit the top boundary, and at least one will hit
the bottom boundary.  In fact, a horizontal and a vertical dimer on
either side of the ray will lie on the upper boundary, and similarly
for the bottom boundary.  \fref{maxMonLimits} is helpful in seeing why
this is true.  These dimers on either side of the ray induce a $000$
in $P$ and another $000$ in $Q$.  (Although not used in this proof,
note that the comments above also apply to bidimers.)  In total, there
are at least $v$ distinct triples of $0$s in each sequence, one for
each vortex.  Now remove $00$ from each triple as in \fref{maxMonSeq}.
The updated $P$ and $Q$ sequences have a combined length of $2c-4v$
and neither of them contains a $11$.  Thus the total number of $1$s is
at most $\lceil |P|/2\rceil + \lceil|Q|/2\rceil$, which is at most
$c-2v+ 1$.  Adding back the $v$ vortex monomers, we conclude that
there are at most $c-v+1$ monomers in total, which finishes the proof.

Note that, to acheive the bound of $c+1$, we must have $v = 0$, and that
the maximum is achieved by a vertical bond pattern.
\end{proof}



 \newsavebox{\tempbox}
 \sbox{\tempbox}{

\begingroup
  \makeatletter
  \providecommand\color[2][]{%
    \errmessage{(Inkscape) Color is used for the text in Inkscape, but the package 'color.sty' is not loaded}
    \renewcommand\color[2][]{}%
  }
  \providecommand\transparent[1]{%
    \errmessage{(Inkscape) Transparency is used (non-zero) for the text in Inkscape, but the package 'transparent.sty' is not loaded}
    \renewcommand\transparent[1]{}%
  }
  \providecommand\rotatebox[2]{#2}
  \ifx\svgwidth\undefined
    \setlength{\unitlength}{216.2pt}
  \else
    \setlength{\unitlength}{\svgwidth}
  \fi
  \global\let\svgwidth\undefined
  \makeatother
  \begin{picture}(1,0.70386216)%
    \put(0,0){\includegraphics[width=\unitlength]{maxMonSeq.eps}}%
    \put(0.69981499,0.6146404){\color[rgb]{0,0,0}\makebox(0,0)[lb]{\smash{$11$}}}%
    \put(0.71091582,0.29271625){\color[rgb]{0,0,0}\makebox(0,0)[lb]{\smash{$1\times $}}}%
    \put(0.55920444,0.36302153){\color[rgb]{0,0,0}\makebox(0,0)[lb]{\smash{$00\cdots$}}}%
    \put(0.55920444,0.04479766){\color[rgb]{0,0,0}\makebox(0,0)[lb]{\smash{$010\cdots$}}}%
    \put(0.29648474,0.36302153){\color[rgb]{0,0,0}\makebox(0,0)[lb]{\smash{$000\cdots$}}}%
    \put(0.29278446,0.6146404){\color[rgb]{0,0,0}\makebox(0,0)[lb]{\smash{$000\cdots$}}}%
    \put(0.29648474,0.29641653){\color[rgb]{0,0,0}\makebox(0,0)[lb]{\smash{$0\times \times \cdots$}}}%
    \put(0.29278446,0.04479766){\color[rgb]{0,0,0}\makebox(0,0)[lb]{\smash{$0\times \times \cdots$}}}%
    \put(0.05190032,0.61402248){\color[rgb]{0,0,0}\makebox(0,0)[lb]{\smash{$T= \cdots$}}}%
    \put(0.05020624,0.36313219){\color[rgb]{0,0,0}\makebox(0,0)[lb]{\smash{$S= \cdots$}}}%
    \put(0.03339893,0.29551179){\color[rgb]{0,0,0}\makebox(0,0)[lb]{\smash{$T= \cdots$}}}%
    \put(0.03826278,0.04377683){\color[rgb]{0,0,0}\makebox(0,0)[lb]{\smash{$S= \cdots$}}}%
  \end{picture}%
\endgroup
}%

\begin{figure}[ht]
  \centering
  \hspace{-5.7cm}
  \subfigure[]{%
    \vbox to \ht\tempbox{ \vfil%
        \includegraphics{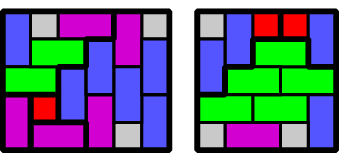}%
        \vfil}%
      \label{fig:maxMonLimits}%
    }%
      \hspace{-5cm}
  \subfigure[]{
    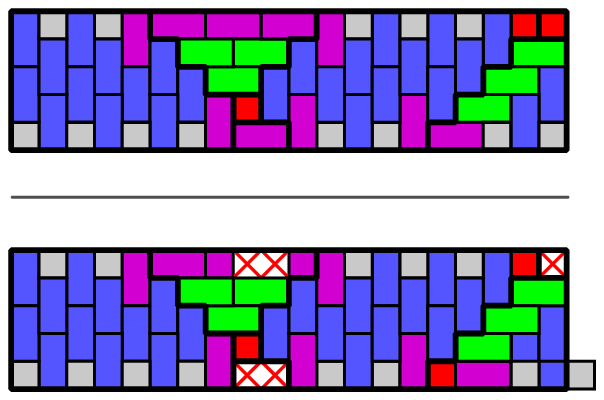    \label{fig:maxMonSeq}}
  \caption{Each vortex and vee is associated with segments of
    monomer-free grid squares shown in purple.  (a) Segments
    associated with vortices have length at least three. Those associated with vees have at
    least two 0s.  (b) The two types of updates to sequences $P$ and $Q$.
    The upper sequences are before the updates and the lower are after updates.
    The symbol $\times $ represents a
    deletion from the sequence.}
  \label{fig:maxMonSequence}
\end{figure}

The converse of \tref{maxmon} is false, for example, Alhazov et
al. (\cite{Alhazov}) show that $T(9,13,1) = 0$.  We now state a couple
of consequences of \tref{maxmon}.

\begin{corollary}\label{cor:maxmonA}
The following three statements are true for tatami tilings of
  an $r \times c$ grid with $r \le c$.
\begin{itemize}
\item[(i)]
The maximum number of monomers in a $r \times c$ grid is $c+1$ if
  $r$ is even and $c$ is odd; otherwise it is $c$.
There is a tatami tiling achieving this maximum.
\item[(ii)]
A tatami tiling with the maximum number of monomers has no vortices.
\item[(iii)]
A tatami tiling with the maximum number of monomers has no bidimers.
\end{itemize}
\end{corollary}
\begin{proof}
  (i) That this is the correct maximum value can be inferred from
  \tref{maxmon}.  A tiling consisting only of vertical running bond
  achieves it, for example.

(ii) This was noted at the end of the proof of \tref{maxmon}.

(iii) We can again use the same sort of reasoning that was used for
vortices in \tref{maxmon}, but there is no need to ``add back'' the
monomers, since bidimers do not contain one.
\end{proof}

\subsection{Square tatami tilings}

In this section, we show that $T(n,n,n) = n2^{n-1}$. \tref{tnnn} relies on the following lemma and corollary.

\begin{lemma}
  \label{lem:flips}
  For each $n\times n$ tiling with $n$ monomers, a trivial tiling can be obtained via a finite sequence of diagonal flips in which each monomer moves at most once.  Reversing this sequence returns us to the original tiling.
\end{lemma}

\begin{proof}

  Let $T$ be the \T-diagram of an $n\times n$ tiling with $n$
  monomers.  Each ray $\rho$ in $T$ touches two adjacent boundaries
  which form a corner $\gamma$, so $\rho$ and $\gamma$ are said to belong to
  each other. For each corner $\gamma$, choose the ray which belongs to it
  and is farthest away from it; if a corner does not have a ray, then
  choose the corner itself. Between the four chosen rays/corners, our
  tiling can only contain either horizontal or vertical running bond
  (by \lref{conditions}).  Let $A(T)$ be the area of this
  \emph{central running bond}.
  
  We begin a sequence of diagonal flips by choosing one ray $\rho$
  that is farthest from its corner and flipping the diagonal $\delta$
  touching $\rho$ that is between $\rho$ and its corner. Let $T'$
  be the resulting \T-diagram.  In $T$, $\delta$ is not part of the central
  running bond and in $T'$, it is; thus $A(T')>A(T)$. Continuing this
  process yields a trivial tiling via a finite sequence of diagonal
  flips.

\begin{figure}
  \centering
    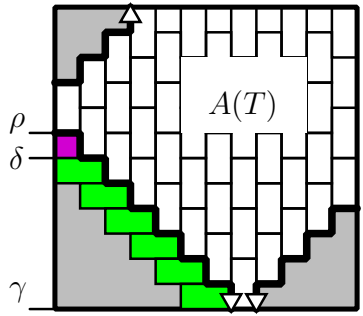
    \caption{In the tiling $T$ from \lref{flips}, the ray $\rho$
      belongs to the corner $\gamma$ and it is associated with the
      diagonal $\delta$. The area $A(T)$ counts the grid squares that
      are not between any ray and its corner. The monomer that is
      moved in the diagonal flip becomes part of $A(T')$ and is
      therefore moved only once in the sequence.  The corner monomers
      are never moved.}
\label{fig:rcorner}
\end{figure}


\end{proof}

\begin{corollary}\label{cor:adjcorn}
Every $n \times n$ tiling with $n$ monomers has two corner monomers and they are in adjacent corners.
\end{corollary}
\begin{proof}
The sequence of diagonals chosen for diagonal flips described in \lref{flips} never includes a diagonal containing a corner monomer because such a diagonal is never between a ray and its associated corner.  As such, the corner monomers are fixed throughout the sequence of diagonal flips yielding a trivial tiling. Since a trivial $n \times n$ tiling with $n$ monomers has two monomers in adjacent corners, then, so must every other $n \times n$ tiling with $n$ monomers.
\end{proof}

\cref{adjcorn} show that the four rotations of any $n \times n$ tiling
with $n$ monomers are distinct.   We call the rotation
with monomers in the top two corners the \emph{canonical case}.


\begin{theorem}\label{thm:tnnn}
  The number of $n\times n$ tilings with $n$ monomers, $T(n,n,n)$, is $n2^{n-1}$.
\end{theorem}
\begin{proof}
  We count the $n\times n$ tilings with $n$ monomers up to rotational symmetry by counting the canonical cases only.  Let $S(n) = T(n,n,n)/4$.  We will give a combinatorial proof that $S(n)$ satisfies the following recurrence:
  \begin{align}
    S(n) = 2^{n-2} + 4 S(n-2)  = n 2^{n-3}  \text{ where } S(1) = S(2)
= 1.
\label{eq:Sn}
  \end{align}
  In \tref{maxmon} we defined a diagonal flip which results in a
  monomer $\mu$ moving up or down depending on the orientation of its
  diagonal.  As such, we simplify our terminology by referring to
  flipping a monomer in a particular direction (up, down, left, or
  right).

  We treat the even and odd cases separately, though the proofs are
  naturally similar. In both cases, we begin with the canonical
  trivial case and consider all possible sequences of flips in which
  each monomer is moved at most once and the corner monomers are fixed.  By
  Lemma 3 and its corollary, this counts the canonical tilings.

  The canonical trivial case for even $n$, shown in
  \fref{evenbasecase:a} for $n=8$, is a horizontal running bond tiling
  with fixed (black) monomers in the top corners and $n/2$ (red and
  yellow) monomers on both the left and right boundaries.  We classify
  the tilings according to what happens to the bottom (yellow) monomer
  on each of these boundaries, which we will call $w$ and $e$.

\begin{figure}[h]
\centering
  \subfigure[]{\includegraphics[scale = 0.7]{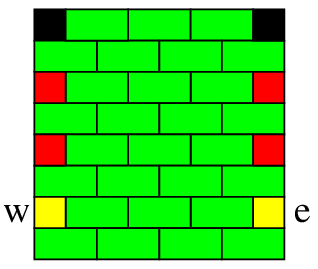}
    \label{fig:evenbasecase:a}
  }
  \subfigure[]{\includegraphics[scale = 0.7]{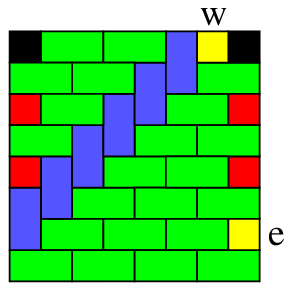}
     \label{fig:evenbasecase:b}
   }
  \subfigure[]{\includegraphics[scale = 0.7]{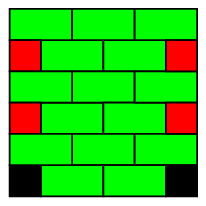}
     \label{fig:evenbasecase:c}
   }
  \subfigure[]{\includegraphics[scale = 0.7]{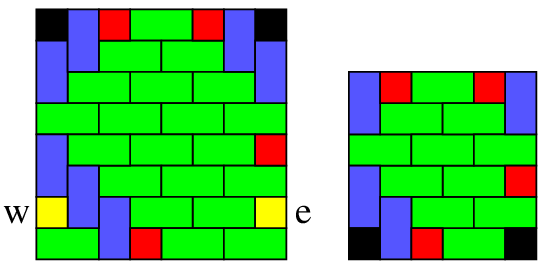}
     \label{fig:evenbasecase:d}
   }
   \caption{(a) Canonical trivial case for an $8\times 8$ square with $8$ monomers. (b) Flipping $w$ up.  (c) 180 degree rotation of the canonical trivial case for a $6\times 6$ square with $6$ monomers. (d) An $8\times 8$ tiling with its associated $6\times 6$ tiling.}
\label{fig:evenbasecase}
\end{figure}

First, suppose $\mu \in \{w,e\}$ is flipped up as shown in \fref{evenbasecase:b}. Because our tiling is square, this flip inhibits any orthogonal diagonal flips and thus the monomers that shared a boundary with $\mu$ before it was flipped up can only be flipped up and monomers on the opposite boundary can only be flipped down. There are $n-3$ such monomers that are not fixed and can be flipped independently of each other. This gives $2^{n-3}$ possibilities when either $w$ or $e$ is flipped up, resulting in a total of $2^{n-2}$ tilings.

If neither $w$ nor $e$ is flipped up, these monomers can be flipped (or not) independently of each other and of other non-fixed monomers, as shown in \fref{even_reddown}. As such, we can now ignore what happens to $w$ and $e$ and consider them fixed, keeping in mind that for each such tiling, there are three others with $w$ and $e$ in different positions.  We will find a one-to-one correspondence between these tilings and one quarter of the $(n-2)\times (n-2)$ tilings with $n-2$ monomers by mapping the monomers of the canonical trivial cases (rotated by 180 degrees in the smaller case) and showing that any sequence of flips in one case can be applied to the equivalent monomers in the other.

\begin{figure}
  \centering
    \includegraphics[scale = 0.7]{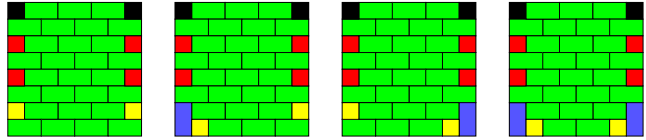}
\caption{The four possibilities for flipping $w$ and/or $e$ down.}
\label{fig:even_reddown}
\end{figure}

There are $n-4$ (red) monomers on the left and right boundaries of the $n \times n$ canonical case that we have not fixed.  Consider the 180 degree rotation of the canonical trivial case for $(n-2)\times (n-2)$ tilings with $n-2$ monomers which has fixed (black) monomers in the \emph{bottom} corners, as shown in \fref{evenbasecase:c} for $n-2=6$. Associate the $n-4$ non-fixed (red) monomers of this tiling with the $n-4$ non-fixed (red) monomers of the $n\times n$ canonical trivial case in a natural way: pairing those in the same position relative to the bottom fixed monomers. Similarly, diagonals containing associated monomers are also associated.

We need to show that compatibility between diagonal flips is preserved between the smaller and larger cases: that is, if two diagonals cannot both be flipped in the larger square, the same is true for the corresponding diagonals in the smaller square, and vice versa.

In both cases, two monomers on the same boundary can both be flipped if and only if they are either flipped in the same direction or the top one is flipped up and the bottom one is flipped down; compatibility is preserved.

For a pair of monomers on opposite boundaries, observe that a conflict between flips can only occur if we try to flip them both in the same direction.  Further, conflict depends entirely on the distance of the monomers from the horizontal centerline of the grid. Let $d_w$ and $d_e$ respectively be the distances from the horizontal centerline, with negative values below the line and positive values above.  If $d_w + d_e > 0$, then the two monomers cannot both be flipped down, and similarly, if $d_w + d_e <0$, they cannot both be flipped up.  This distance is preserved between the associated monomers in the larger and smaller squares and thus compatibility is also preserved.

There are $S(n-2)$ ways of flipping the monomers of the (rotated) $(n-2)\times (n-2)$ canonical trivial case, and thus $S(n-2)$ ways of flipping the corresponding monomers of the $n\times n$ canonical trivial case. This yields $4S(n-2)$ tilings, one for each way of positioning $w$ and $e$ and establishes (\ref{eq:Sn}) for even $n$.

The canonical trivial case for odd $n$, shown in \fref{oddbasecase:a} for $n=7$, is a vertical running bond tiling with (black) monomers in the top corners.  It has $\lceil n/2\rceil$ monomers on the top boundary and $\lfloor n/2 \rfloor$ monomers on the bottom boundary.  Label the bottom left and bottom right monomers $w$ and $e$ respectively.

\begin{figure}
  \centering
  \subfigure[]{\includegraphics[scale = 0.7]{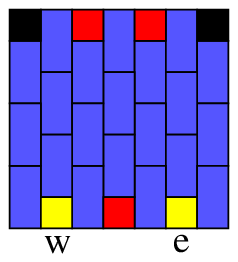}
    \label{fig:oddbasecase:a}
  }
  \subfigure[]{\includegraphics[scale = 0.7]{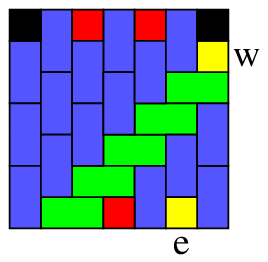}
     \label{fig:oddbasecase:b}
   }
  \subfigure[]{\includegraphics[scale = 0.7]{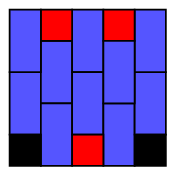}
     \label{fig:oddbasecase:c}
   }
  \subfigure[]{\includegraphics[scale = 0.7]{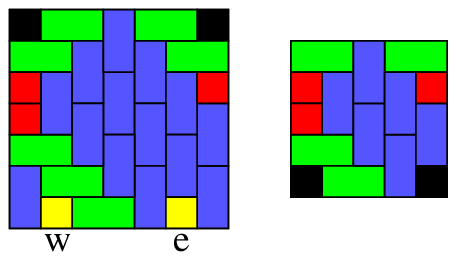}
     \label{fig:oddbasecase:d}
   }
   \caption{(a) Canonical trivial case for a $7\times 7$ square with $7$ monomers. (b) Flipping $w$ to the right.  (c) Canonical trivial case for a $5\times 5$ square with $5$ monomers. (d) An $7\times 7$ tiling with its associated $5\times 5$ tiling.}
\label{fig:oddbasecase}
\end{figure}

Similar to the even case, if either $w$ is flipped right (as in \fref{oddbasecase:b}) or $e$ is flipped left, there are $n-3$ monomers which can be flipped independently to obtain other tilings and this yields $2^{n-2}$ tilings.

Otherwise, $w$ and $e$ can be flipped left and right (respectively) independently of each other and of other monomers, as shown in \fref{odd_reddown}.  Again we fix $w$ and $e$, keeping in mind that for each such tiling, there are three others with $w$ and $e$ in different positions.  We will find a similar one-to-one correspondence to the one in the even case.

\begin{figure}
  \centering
    \includegraphics[scale = 0.7]{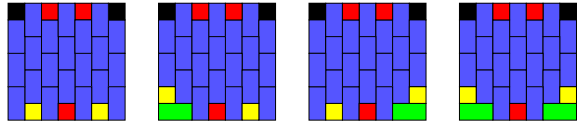}
    \caption{The four possibilities for flipping (or not flipping) $w$ and $e$ to the left and right respectively.}
\label{fig:odd_reddown}
\end{figure}

There are $n-4$ (red) monomers on the top and bottom boundaries of the
canonical trivial case that we have not fixed. Once again, we associate these monomers with those of the 180 degree rotation of the $(n-2)\times (n-2)$ canonical trivial case with $n-2$ monomers which has fixed (black) monomers in the \emph{bottom} corners. Arguing as in the even case, with a vertical centerline rather than a horizontal one, we conclude that a pair of monomers in the $n\times n$ tiling can be flipped if and only if the corresponding flips can be made in the $(n-2)\times (n-2)$ tiling. Again this yields $4(S(n-2))$ tilings and establishes (\ref{eq:Sn}) for odd $n$.
\end{proof}

\subsection{Fixed height tatami tilings}
\label{sec:fixedheight}

In this section we show that for a fixed number of rows $r$, the
ordinary generating function of the number of tilings of an $r \times
c$ rectangle is a rational function. We will show that, for each value
of $r$, the number of fixed-height tilings satisfies a system of
linear recurrences with constant coefficients. We will derive the
recurrences for small values of $r$ and then discuss an algorithm
which can be used for larger values of $r$.

\vspace{0.3cm}
 Let $T_r(z)$ denote the generating function
 \begin{align*}
T_r(z) = \sum_{c \ge 0} T(r,c) z^c.
\end{align*}
For $c\geq 2$, a tatami tiling of a $1\times c$ rectangle begins with
either a monomer or a dimer. Thus, $T(1,c)=T(1,c-1)+T(1,c-2)$ for
$c\geq 2$, where $T(1,0)=1$ and $T(1,1)=1$. This is the well known
Fibonnaci recurrence. Since it is a linear recurrence with constant
coefficients, it is not a difficult task to verify that
\begin{align*}
T_1(z)=\frac{1+z}{1-z-z^2}.
\end{align*}

For each $r\ge 2$ we derive a recurrence based on partial tilings
which can be solved using mathematical software such as Maple. A
partial tiling of an $r\times c$ grid is a \emph{minimal} $r\times k$
\emph{tiling} if and only if the first $k$ columns are covered and no
tile lies entirely outside of these columns.  The $r$ and $k$ may
sometimes be omitted. Let $\mathcal{S}_r$ be the set of configurations
which correspond to a minimal $r\times 1$ tiling. For $s_{v}\in
\mathcal{S}_r$, let $v$ be a ternary $r$-tuple whose elements
correspond to the grid squares of the first column, ordered from top
to bottom.  The elements $0,1$, and $2$, each represent a grid square
covered by a vertical dimer, monomer, or horizontal dimer,
respectively.  Note that $0$s always appear in consecutive pairs. For
example, $s_{0012002}\in \mathcal{S}_7$ corresponds to the minimal
$7\times 1$ tiling shown in \fref{7example:a}.

\begin{figure}[ht]
\centering
  \subfigure[]{ \includegraphics{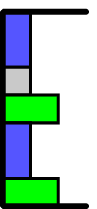}
    \label{fig:7example:a}
  }
\hspace{2cm}
  \subfigure[]{ \includegraphics{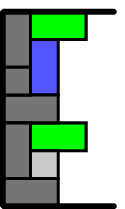}
    \label{fig:7example:b}
  }
  \hspace{2cm}
  \subfigure[]{ \includegraphics{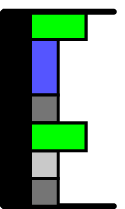}
    \label{fig:7example:c}
  }
  \caption{(a) The minimal $7\times 1$ tiling corresponding to
    $s_{0012002}$.  (b) A possible extension of the minimal tiling in
    (a).  (c) Removing the first column yields a new minimal tiling,
    represented by $s_{2001211}$. }
  \label{fig:7example}
\end{figure}

For $c\geq 1$, let $t_{r,v}(c)$ be the number of tilings of an
$r\times c$ rectangle that begin with the minimal $r\times 1$ tiling
$s_{v}$. Note that $t_{r,v}(1)=1$ if $v$ does not contain a 2, and is
zero otherwise. To derive a recurrence we determine all ways of
extending each configuration in $s_{v}$ to a minimal $r\times 2$
tiling.  By taking each of these minimal $r\times 2$ tilings and
chopping off the first column we can match these tilings to an element
in $\mathcal{S}_r$ which will define a recurrence.  \fref{7example:b}
shows an extension of the tiling $s_{0012002}$ and \fref{7example:c}
shows that this extension corresponds to the configuration
$s_{2001211}$.  Notice that \fref{7example:c} can only be extended
once more.

\begin{lemma} \label{r=2} $T_2(z)=\frac{1+2z^2-z^3}{1-2z-2z^3+z^4}$.
\end{lemma}

\newcommand{\ig}[1]{\vspace{0.15cm}\includegraphics[scale=0.75]{#1}}
\newcommand{\igs}[1]{\vspace{0.15cm}\includegraphics[scale=0.5]{#1}}
\begin{proof}

  For $r=2$, we have
  $\mathcal{S}_2=\{s_{00},s_{11},s_{12},s_{21},s_{22}\}$. Since
  $\mathcal{S}_2$ contains all possible ways to start a tiling of an
  $r\times c$ rectangle, with $c\geq 2$, it follows that
\begin{equation}
T(2,c)=t_{00}(c)+t_{11}(c)+t_{12}(c)+t_{21}(c)+t_{22}(c).
\label{eq:r=2}
\end{equation}
The initial conditions are
$t_{00}(1)=1,t_{11}(1)=1,t_{12}(1)=0,t_{21}(1)=0$, and $t_{22}(1)=0$. To
derive the recurrence we consider the ways of extending each of the
minimal $2\times 1$ tilings in $\mathcal{S}_2$ to a minimal $2\times
2$ tiling.

\newcolumntype{S}{>{\arraybackslash} m{.1\textwidth} }
\newcolumntype{E}{>{\arraybackslash} m{.3\textwidth} }
\newcolumntype{R}{>{\arraybackslash} m{.4\textwidth} }

\begin{center}
\begin{tabular}{c|S|E|R}
   & Start & Extensions & Recurrences\\
  \hline
  \hline
  $S_{00}$ & \ig{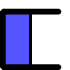} & \ig{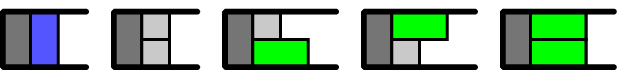} & \ig{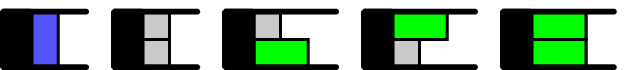}\\
  \hline\multicolumn{4}{c}{}\\
  \multicolumn{2}{c}{}&\multicolumn{2}{l}{$t_{00}(c)=t_{00}(c-1)+t_{11}(c-1)+t_{12}(c-1)+t_{21}(c-1)+t_{22}(c-1)$}\\ \multicolumn{4}{c}{}\\\hline\hline
  $S_{11}$ & \ig{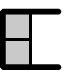} & \ig{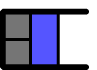} & \ig{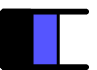}\\
  \hline\multicolumn{4}{c}{}\\
  \multicolumn{2}{c}{}&\multicolumn{2}{l}{$t_{11}(c)=t_{00}(c-1)$}\\ \multicolumn{4}{c}{}\\\hline\hline
  $S_{12}$ & \ig{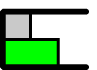} & \ig{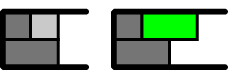} & \ig{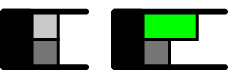}\\
  \hline\multicolumn{4}{c}{}\\
  \multicolumn{2}{c}{}&\multicolumn{2}{l}{$t_{12}(c)=t_{11}(c-1)+t_{21}(c-1)$}\\ \multicolumn{4}{c}{}\\\hline\hline
  $S_{21}$ & \ig{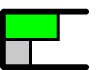} & \ig{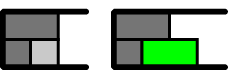} & \ig{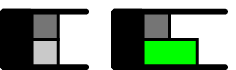}\\
  \hline\multicolumn{4}{c}{}\\
  \multicolumn{2}{c}{}&\multicolumn{2}{l}{$t_{21}(c)=t_{11}(c-1)+t_{12}(c-1)$}\\ \multicolumn{4}{c}{}\\\hline\hline
  $S_{22}$ & \ig{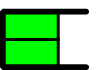} & \ig{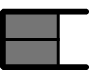} & \ig{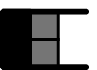}\\
  \hline\multicolumn{4}{c}{}\\
  \multicolumn{2}{c}{}&\multicolumn{2}{l}{$t_{22}(c)=t_{11}(c-1)$}
\end{tabular}
\end{center}

 By solving the system of recurrences defined by these five cases and
 \eref{r=2} we arrive at the stated result for $r=2$.
\end{proof}

The process outlined in the proof of Lemma \ref{r=2} can be
implemented in an algorithm. We determine the set $\mathcal{S}_r$ with
an exhaustive search.  Then, for each element $s_{v}\in
\mathcal{S}_r$, we list all extensions to a minimal $r\times 2$
tiling. Each extension of $s_{v}$ produces a unique sum-term in the
recurrence for $t_{v}(c)$.  Once again, the initial conditions are
\begin{align*}
  t_{v}(1)=\left\{
    \begin{array}{ll}
      1,~~ &\text{if} ~ v~\text{contains a}~2,\\
      0,~~&\text{otherwise}.
    \end{array}
  \right.
\end{align*}

We may reduce the number of equations in the system of recurrences by
ignoring elements of $\mathcal{S}_r$ which cannot be extended to a
minimal $r\times 2$ tiling.  This may be taken further by determining
necessary conditions for an element of $\mathcal S_r$ to be extendible
to an $r\times c$ tiling.

The algorithm produces a system of linear recurrences with constant
coefficients. This proves the following result.

\begin{theorem} For a fixed number of rows $r$, the ordinary
  generating function for the number of tilings of an $r\times n$
  rectangle is a rational generating function.
\end{theorem}

The output of our algorithm for $r=3$ gives the following generating function:

$$T_3(z)=\frac{1+2z+8z^2+3z^3-6z^4-3z^5-4z^6+2z^7+z^8}{1-z-2z^2-2z^4+z^5+z^6}.$$

It is impractical to include the complete generating function for
any larger values of r. The degrees for the numerators and
denominators, however, as well as the coefficients in the denominators are
given in \tabref{Genfuntable} for $r=1,2,...,11$. The salient patterns
in these coefficients are summarized in Conjectures
\ref{conj:selfreciprocal} and \ref{conj:degdenoms}.  Note
that \conref{selfreciprocal}$(a)$ implies $g(z)$ is a self-reciprocal
polynomial for $r\equiv 2 \pmod{4}$.

\begin{conjecture} \label{conj:selfreciprocal} Let
  $T_r(z)=\frac{f(z)}{g(z)}$, where $f(z)$ and $g(z)$ are relatively
  prime polynomials, and $\deg (g(z))=n$, and $r\geq 1$.
  Then,
  \begin{align*}
    g(z)=\left\{\begin{aligned}
        -z^ng&\left(\,\,\,\, \frac{1}{z}\right),& \text{if}~ r\equiv 0 \pmod{4},\\
        -z^ng&\left(- \frac{1}{z}\right),& \text{if}~ r\equiv 1 \pmod{4},\\
        z^ng&\left(\,\,\,\, \frac{1}{z}\right),& \text{if}~ r\equiv 2 \pmod{4},\\
        z^ng&\left(- \frac{1}{z}\right),& \text{if}~ r\equiv 3
        \pmod{4}.
      \end{aligned}\right.
  \end{align*}
\end{conjecture}

A mod $4$ pattern also seems to occur in the degrees of the
denominators of $T_r(z)$.  The rigid structure we encounter in tatami
tilings prompts us to infer this pattern upon all values as well.

\begin{conjecture}
  \label{conj:degdenoms}
  Let $g(z)$ be the denominator of $T_r(z)$.  Then,
  \begin{align*}
    \deg( g(z) ) = \left\{
      \begin{aligned}
        8m^2+2m+1, && \text{if}~ r\equiv 0 \pmod{4},\\
        8m^2+4m+2, && \text{if}~ r\equiv 1 \pmod{4},\\
        8m^2+10m+4, && \text{if}~ r\equiv 2 \pmod{4},\\
        8m^2+8m+6, && \text{if}~ r\equiv 3 \pmod{4}.
      \end{aligned}\right.
  \end{align*}
\end{conjecture}

\begin{table}[htbp]
  \centering \scalebox{0.85}{%
    \begin{tabular}{|c|c|c|l|}\hline
      $r$ & $p$ & $q$ &
      \begin{minipage}[h]{.85\linewidth}
        \vspace{0.1cm} Coefficients of $g(z)$ ordered from left to
        right by ascending degree and then folded like these arrows;
        \includegraphics[scale=1.0]{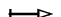} for $r\le
        3$, \includegraphics[scale=1.0]{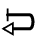} for
        $r=4,5,6,7$, and \includegraphics[scale=1.0]{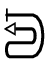} for  $r\ge 8$.
        \vspace{0.1cm}
      \end{minipage} \\ \hline

      $1$&1&2&\texttt{~1,-1,1}\\	\hline
      $2$&3&4&\texttt{~1, -2, 0, -2, 1}\\ \hline
      $3$&$8$&$6$&      \texttt{~1, -1, -2, 0, -2, 1, 1}   \\ \hline
      $4$&$14$&$11$&
      \begin{minipage}{0.85\linewidth}
        \vspace{0.1cm}

        \texttt{-1,~1,~1,~1,-1,~7}

        \texttt{~1,-1,-1,-1,~1,-7}

        \vspace{0.1cm}
      \end{minipage}  \\
      \hline
      $5$&$18$&$14$&
      \begin{minipage}{0.85\linewidth}
        \vspace{0.1cm}

        \texttt{-1,~1,~1,-1,~3,-1,~5,-2}

        \texttt{~1,~1,-1,-1,-3,-1,-5}

        \vspace{0.1cm}
      \end{minipage}  \\
      \hline
      $6$&$27$&$22$&
      \begin{minipage}{0.85\linewidth}
        \vspace{0.1cm}

        \texttt{~1,-1,-1,~1,-1,-2,~2,-10,~9,-1,~4, 6}

        \texttt{~1,-1,-1,~1,-1,-2,~2,-10,~9,-1,~4}

        \vspace{0.1cm}
      \end{minipage}  \\
      \hline
      $7$&$28$&$22$&
      \begin{minipage}{0.85\linewidth}
        \vspace{0.1cm}

        \texttt{~1,-1,-3,~3,~4,-4,-9,~7,~6,-5,~2, 0}

        \texttt{~1,~1,-3,-3,~4,~4,-9,-7,~6,~5,~2}

        \vspace{0.1cm}
      \end{minipage}  \\
      \hline
      $8$&$44$&$37$&
      \begin{minipage}{0.85\linewidth}
        \vspace{0.1cm}

        \texttt{-1,~1,~1,-1,~1,-1,~1,~3,-3,~13,-12}

        \texttt{~1,-1,-1,~1,-1,~1,-1,-3,~3,-13,~12}

        \texttt{~34,-2,~6,~20,-6,~12, 0, 0}

        \texttt{-34,~2,-6,-20,~6,-12, 0, 0}

        \vspace{0.1cm}
      \end{minipage}  \\
      \hline
      $9$&$50$&$42$&
      \begin{minipage}{0.85\linewidth}
        \vspace{0.1cm}

          \texttt{-1,~1,~1,-1,~1,-1,~1,-1,~5,-3,~11,-8}

          \texttt{~1,~1,-1,-1,-1,-1,-1,-1,-5,-3,-11,-8}

          \texttt{10,~24,~2,~28,~2,~20,~8,~14,~4,~6}

          \texttt{\hspace{19.5pt}-24,~2,-28,~2,-20,~8,-14,~4,-6}

        \vspace{0.1cm}
      \end{minipage}  \\
      \hline
      $10$&$65$&$56$&
      \begin{minipage}{0.85\linewidth}
        \vspace{0.1cm}
        \footnotesize{

          \texttt{~1,-1,-1,~1,-1,~1,-1,~1,-1,-4,~4,-16,~15,~1,-1}

          \texttt{~1,-1,-1,~1,-1,~1,-1,~1,-1,-4,~4,-16,~15,~1,-1}

          \texttt{-120,~68,-78,-18,~18,-66,~66,-2,~7,~41,-23,~33,-17,~17}

          \texttt{\hspace{32.5pt}68,-78,-18,~18,-66,~66,-2,~7,~41,-23,~33,-17,~17}

        }
\vspace{0.1cm}
      \end{minipage}  \\
      \hline
      $11$&$64$&$54$&
      \begin{minipage}{0.85\linewidth}
        \vspace{0.1cm}
        \footnotesize{

          \texttt{~1,-1,-5,~5,~13,-13,-27,~27,~48,-48,-83,~81,~125,-120,-160}

          \texttt{~1,~1,-5,-5,~13,~13,-27,-27,~48,~48,-83,-81,~125,~120,-160}

          \texttt{-34,~83,~89,-156,-165,~199,~210,-202,-206,~185,~193,-154}

          \texttt{-34,-83,~89,~156,-165,-199,~210,~202,-206,-185,~193,~154}

          \vspace{0.1cm}
        }
      \end{minipage}  \\
      \hline
    \end{tabular}}
  \caption{Summary of generating function attributes for fixed height
    tilings, $r=1,...,11$, and where $p$ and $q$ are the degrees of
    the numerator and denominator of $T_r(z)$, respectively. The
    degree ordering shows the patterns of \conref{selfreciprocal}.
  }
  \label{tab:Genfuntable}
\end{table}

\section{More conjectures and further research}


The \T-diagram structure removes much of the mystery from tatami tilings and
motivates considerable future work.  In this section we list some open
problems and conjectures, beginning with another counting problem on
rectangular grids.

\subsection{Rectangular regions}
\begin{conjecture}
For all $d \ge 0$ and $m \ge 1$ there is an $n_0$ such
that, for all $n \ge n_0$.
\begin{align*}
T(n,n+d,m) = T(n_0,n_0+d,m),
\end{align*}
whenever $n(n+d)$ has the same parity as $m$ (otherwise $T(n,n+d,m)
= 0$, by \tref{maxmon}).
\end{conjecture}

Experimentally, it appears that the smallest $n_0$ is $m+d+4$, if
$d\ge 1$.

The easiest case occurs when $d = 0$ and $m = 1$.  It is
not hard to show that for all odd $n \ge 3$ we have $T(n,n,1) = 10$
(the single monomer must go at a corner or in the center).
  
In a subsequent paper we will show that for $m<n$,
\begin{align*}
  T(n,n,m) = m2^m +(m+1)2^{m+1},
\end{align*}
whenever $m$ and $n$ have the same parity.

Returning to the subject of generating functions, ignoring signs, it
appears that the denominators of $T_r(z)$ in \sref{fixedheight} are
self-reciprocal.  There must be a combinatorial explanation for this.
Similar questions in the non-tatami case are considered in
\cite{ABBP}.

Generating functions also appear in \conref{cyclotomic}, inspired by
conversations with Knuth.  Let $T(n,z)$ be the generating polynomial
for the number of $n\times n$ tilings with $n$ monomers and $i$
vertical dimers.  Once again, to count such tilings we consider
flipping diagonals, with the added precaution that the sum of the
number of tiles in the flipped diagonals is a given constant.  The
relationship between this and subsets of $\{1, \ldots, n\}$ which have
a given sum is detailed in a subsequent publication.

Let $\phi_n(z)$ denote the $n$th cyclotomic polynomial.  Recall that
the roots of $\phi_n(z)$ are the primitive roots of unity.  One of
their more well-known properties is that
\begin{equation}
1-z^n = \prod_{d \mid n} \phi_d(z).
\end{equation}

Let $S_n(z)$ denote the ordinary generating function of subsets of
$\{1, \ldots, n\}$ which have a given sum.  That is, $\langle z^k
\rangle S_n(z)$ is the number of subsets $A$ of \\
$\{1,2,\ldots,n\}$
such that the sum of the numbers in $A$ is $k$.  It is not difficult
to see that
\begin{align}
S_n(z) = (1+z)(1+z^2) \cdots (1+z^n) = \prod_{k=1}^n (1+z^k).
\label{eq:Snz}
\end{align}

Let $\nu(n)$ denote the number of 2s in the prime factorization of $n$
and note that
\begin{align*}
  1+z^n = \frac{1-z^{2n}}{1-z^n} = \frac{\prod_{d \mid 2n}
    \phi_d(z)}{\prod_{d \mid n} \phi_d(z)} = \prod_{\substack{d \mid 2n \\ d
    \nmid n}} \phi_d(z) = \prod_{\substack{d \mid n \\ d \text{ odd}}}
  \phi_{2^{1+\nu(n)}d}(z).
\end{align*}

When this latter expression is used in $S_n(z)$ some interesting
simplification occurs.
\begin{lemma}
For all $n \ge 1$,
\begin{align*}
S_n(z) = \prod_{j=1}^n \left( \phi_{2j} (z)\right)^{\lfloor \frac{n+j}{2j} \rfloor}  
\end{align*}

\label{lem:nj2}
\end{lemma}

\begin{proof}
  The index $2j$ will occur for those $k$s in \eref{Snz} for which $j
  = 2^{\nu(k)}d$ for some odd $d$ where $d \mid k$.  This equation is
  satisfied for $k = j,3j,5j,\ldots$.  There are $\lfloor (n+j)/(2j)
  \rfloor$ such $k$s that are less than or equal to $n$ (See
  \fref{floor}).
\end{proof}

\begin{figure}[ht]
\centering
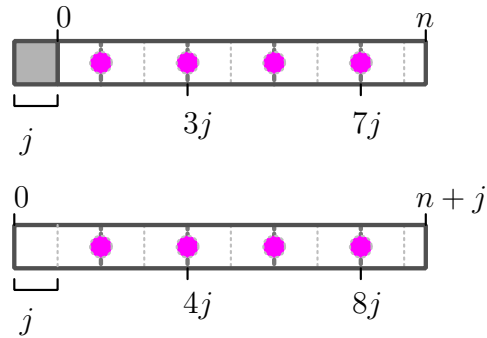
\caption{A visual aid for the last line of the proof of \lref{nj2}.
  The pink dots represent the sequence, $j,3j,5j, \ldots$, with $ij\le
  n$.  Adding $j$ to $n$ shows that the number of dots is $\lfloor
  (n+j)/(2j) \rfloor$. }
\label{fig:floor}
\end{figure}

\begin{conjecture}
\label{conj:cyclotomic}
  The generating polynomial $T(n,z)$ has the factorization
  \begin{align*}
    T(n,z) = P(n,z)\prod_{j\ge 1} S_{\left\lfloor \frac{n-1}{2^j}
      \right\rfloor}(z)
  \end{align*}
  where $P(n,z)$ is an irreducible polynomial.
\end{conjecture}

  We return to the topic mentioned in the introduction: Tatami-tilings
  of orthogonal regions.

\subsection{Orthogonal regions}

\begin{figure}[ht]
  \centering
  \subfigure[]{
    \includegraphics[scale=0.45]{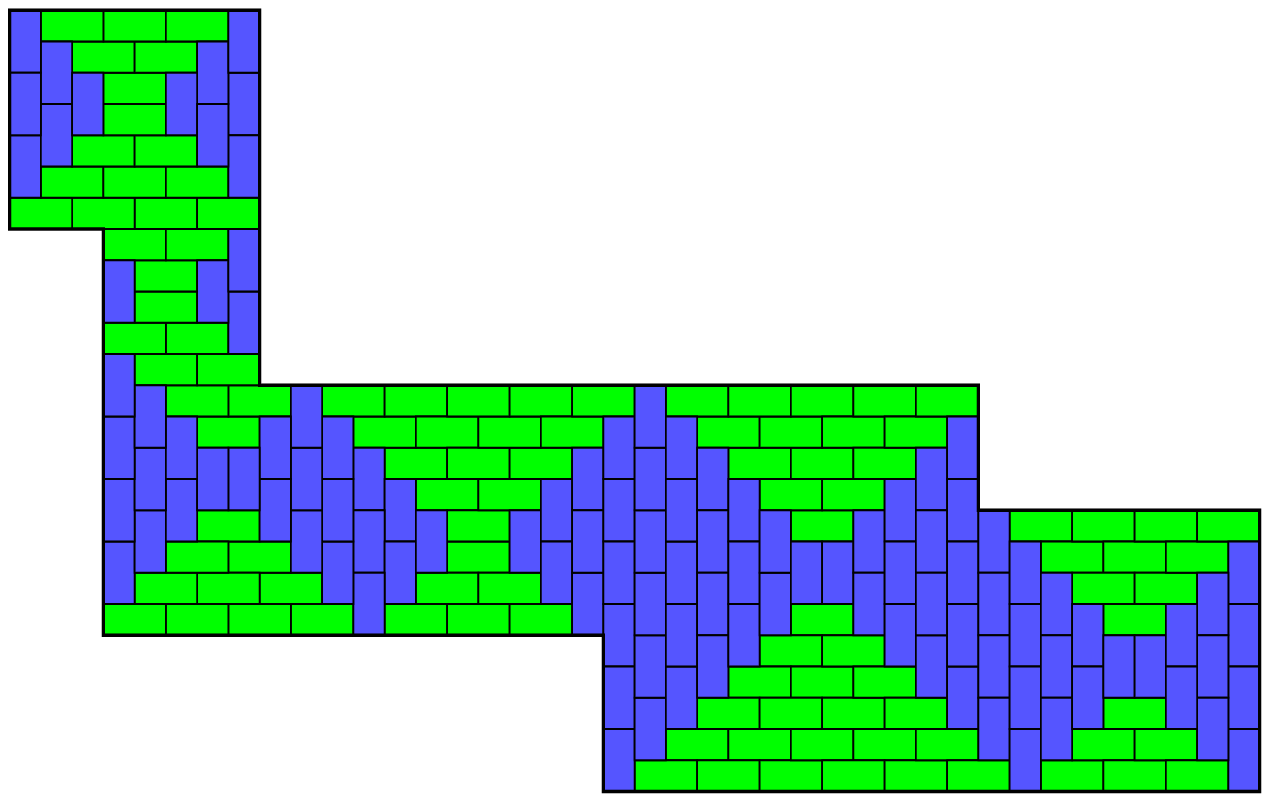}
\label{fig:solution:a}
  }
  \subfigure[]{
    \includegraphics[scale=0.45]{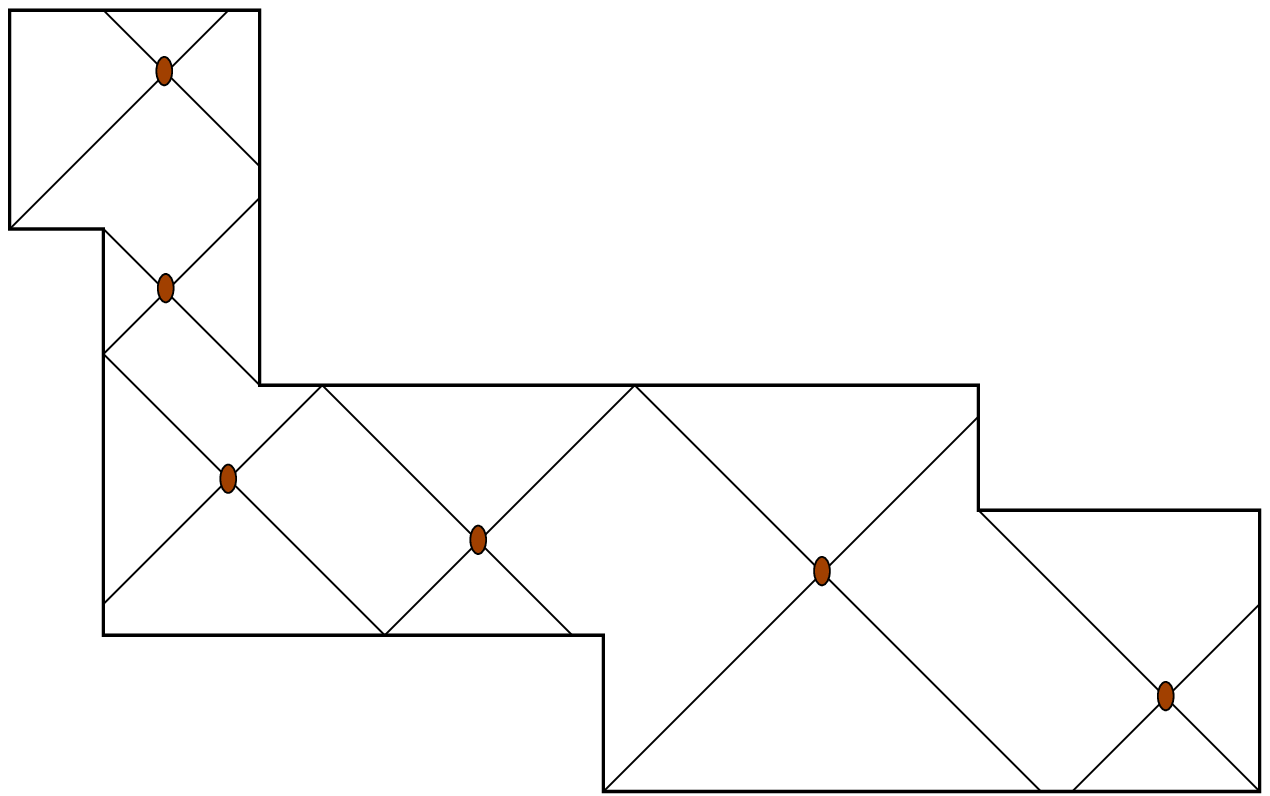}
    \label{fig:solution:b}
  }
  \caption{(a) The solution to the question posed in \fref{driveway}; no
    monomers are required to tatami tile the region.  (b) A legal
    configuration of six magnetic water striders in an orthogonal
    ``pond''.  Note that no further striders may be added.}
\label{fig:solution}
\end{figure}

We believe that the main structural components are the same as they
were for rectangles, but there are a few subtleties to be clarified at
inside corners, since a ray could begin at such a place.

What is the computational complexity of determining the least number
of monomers that can be used to tile an orthogonal region given the
segments that form the boundary of the region and the unit size of
each dimer/monomer?  In the rectangular grid this is answerable in
polynomial time using \T-diagrams, however, it appears to be NP-hard
for an arbitrary number of segments.

The problem of minimizing the number of monomers in a tiling inspires
what we call the ``magnetic water strider problem''.  This time the
orthogonal region is a pond populated by water striders.  A water
strider is an insect that rides atop water in ponds by using surface
tension.  Its 4 longest legs jut out at 45 degrees from its body.  In
the fancifully named magnetic water strider problem, we require the
body to be aligned north-south.  Furthermore its legs support it, not
by resting on the water, but by extending to the boundary of the pond.
Naturally, the legs of the striders are not allowed to intersect.  A
legal configuration of magnetic water striders in an orthogonal pond
is shown \fref{solution:b}.

There are two problems and a game here.  The first is a packing
problem: What is the largest number of magnetic waters striders that a
pond can support?  On the other hand, one can ask what is the minimum
number that can be placed so that no more can be added.  Placing and
packing striders can be tricky, which gives rise to an adversarial
game where players take turns placing striders in an orthogonal
region.  Brian Wyvill has kindly implemented a version of this game,
available at \url{http://www.theory.cs.uvic.ca/~cos/tatami/}.

Interpreted as a matching problem on a subgraph of a grid graph $G$, a
tatami tiling is a matching $M$ with the property that $G-M$ contains
no 4-cycles.  Note that there is always such a matching (e.g., take
the ``running bond'' layout on the infinite grid graph and then
restrict it to $G$).  However, if we insist on a perfect matching,
then the problem is equivalent to our ``perfect'' driveway paving
problem from the introduction.

More generally, a matching whose removal destroys $k$-cycles is called
\emph{$C_k$-transverse}.  Ross Churchley proved that finding a
$C_k$-transverse matching in an arbitrary graph is NP-hard when $k\ge 4$
(private communication \cite{churchley}).

\subsection{Combinatorial games}

Consider the following game.  Given an orthogonal region, players take
turns placing dimers (or dimers and monomers); each placement must
satisfy the tatami constraint and the last player who can move wins.
This game, called Oku!, is reminiscent of the game called Nimm, in
which players also win by making the last move, however a winning
strategy for our game is unknown and there are grid sizes in which the
second player can force a win.  The name is a phonetic spelling of the
Japanese word for ``put''.

Another game applies tomography to rectangular tilings.

Tiling tomography is a rich and open area of complexity theory to
which a good introduction can be found in \cite{chrobak}.  The
relevant question is as follows: Given $r+c$ triples of numbers
$(h,v,m)$, one for each row and one for each column, is there a tatami
tiling which has $h$ horizontal dimers, $v$ vertical dimers, and $m$
monomers in the respective row or column?  

Without the tatami condition this decision problem is NP-hard (Theorem
4, \cite{DGM}).  Hard or not, the tatami condition gives considerable
information in practice, however, making the reconstruction of a
tatami tiling an entertaining challenge.  Erickson, A., has created an
online computer game out of this called Tomoku. It is playable at
\url{http://tomokupuzzle.com}, complete with music, countdown timers
and high scores.

\begin{figure}[ht]
  \centering 
\includegraphics[width = 0.7\textwidth]{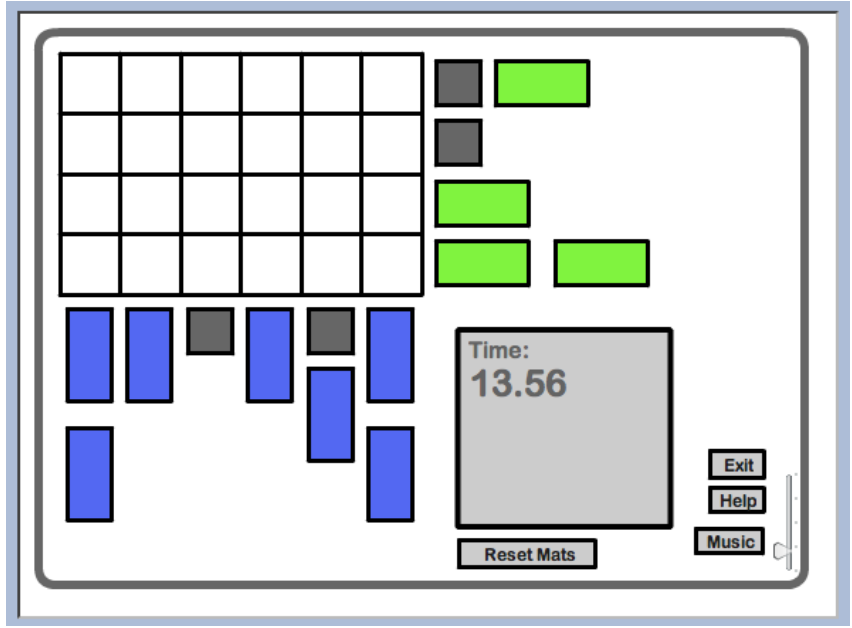}
\caption{The Tomoku web game.  The player is shown which tiles are
  completely contained in each row and column, and the object is to
  reconstruct the tiling.  Note that each monomer appears twice in the
  projections.}
\label{fig:tomoku}
\end{figure}

Both of these games are at
\url{http://www.theory.cs.uvic.ca/~cos/tatami/}, and it should be
noted that they can also be played with a pencil and paper.

\section{Acknowledgements}
Thanks to Donald Knuth for his comments on an earlier draft of this
paper and to Martin Matamala for pointing out the tomography problem.


\end{document}